\documentclass[12pt]{article}

\usepackage{a4wide}
\usepackage{graphicx}
\usepackage{amsmath}
\usepackage{amssymb}
\usepackage{amsthm}
\usepackage{enumerate}
\usepackage{hyperref}

\newtheorem{theorem}{Theorem}[section]

\newtheorem{proposition}[theorem]{Proposition}
\newtheorem{lemma}[theorem]{Lemma}
\newtheorem{claim}[theorem]{Claim}
\newtheorem{corollary}[theorem]{Corollary}
\newtheorem{conjecture}[theorem]{Conjecture}

\usepackage[margin=20pt,small,bf]{caption}
\newcommand{\ILF}[2]{{#1}/{#2}}
\newcommand{\kPartite}[2]{K_{#1} ^{#2}}

\begin{document}

\title{Calculating Ramsey Numbers by partitioning coloured graphs}

\author{\large{Alexey Pokrovskiy} \thanks{Research supported by the LSE postgraduate research studentship scheme.} 
\\
\\Department of Mathematics,
\\ {Freie Universit\"at,} 
\\ Berlin, Germany. 
\\ {Email: \texttt{alja123@gmail.com}}
}

\maketitle

\begin{abstract}

In this paper we prove a new result about partitioning coloured complete graphs and use it to determine certain Ramsey Numbers exactly.
The partitioning theorem we prove is that for $k\geq 1$,  in every edge colouring of $K_n$ with the colours red and blue, it is possible to cover all the vertices with $k$ disjoint red paths and a disjoint blue balanced complete $(k+1)$-partite graph.  
When the colouring of $K_n$ is connected in red, we prove a stronger result---that it is possible to cover all the vertices with $k$ red paths and a blue balanced complete $(k+2)$-partite graph.  

Using these results we determine the Ramsey Number of a path, $P_n$, versus a balanced complete $k$-partite graph, $\kPartite{m}{t}$, whenever $m\equiv 1 \pmod{n-1}$.  We show that in this case $R(P_n, \kPartite{m}{t})= (t-1)(n-1)+t(m-1)+1$, generalizing a result of Erd\H{o}s who proved the $m=1$ case of this result.
We also determine the Ramsey Number  of a path $P_n$ versus the power of a path $P^t_n$.  We show that $R(P_n, P_n^t)=t(n-1)+\left\lfloor \frac{n}{t+1} \right\rfloor$, solving a conjecture of Allen, Brightwell, and Skokan.
\end{abstract}

\section{Introduction}
Ramsey Theory is a branch of mathematics concerned with finding ordered substructures in a mathematical structure which may, in principle, be highly disordered.  An early example of a result in Ramsey Theory is a theorem due to Van der Waerden \cite{VDW}, which says that for for any $k$ and $r\geq 1$ there is a number $W(k, r)$, such that any colouring of the numbers $1,2, \dots, W(k,r)$ with $r$ colours contains a monochromatic $k$-term arithmetic progression.  
A special case of a theorem due to Ramsey \cite{Ramsey} says that for every $n$, there exists a number $R(n)$, such that every 2-edge-coloured complete graph on more than $R(n)$ vertices contains a monochromatic complete graph on $n$ vertices.  The number $R(n)$ is called a \emph{Ramsey number}.

A central definition in Ramsey Theory is the \emph{generalized Ramsey number} $R(G)$ of a graph~$G$: the minimum $n$ for which every 2-edge-colouring of $K_n$ contains a monochromatic copy of~$G$.    For a pair of graphs $G$ and $H$ the Ramsey number of $G$ versus $H$, $R(G,H)$, is defined to be the minimum $n$ for which every 2-edge-colouring of $K_n$ with the colours red and blue contains either a red copy of $G$ or a blue copy of $H$. 
Although there have been many results which give good bounds on Ramsey numbers of graphs \cite{RamseyBook}, the exact value of the Ramsey number $R(G,H)$ is only known when $G$ and $H$ each belong to one of a few families of graphs.

One of the first Ramsey numbers to be determined exactly was the Ramsey number of the path.
\begin{theorem}[Gerencs\'er and Gy\'arf\'as, \cite{Gerencser}]\label{PathRamsey}
For $m\leq n$ we have that
$$R(P_n,P_m)= n+\left\lfloor\frac{m}{2}\right\rfloor-1.$$
\end{theorem}

In the same paper where Gerencs\'er and Gy\'arf\'as proved Theorem~\ref{PathRamsey}, they also proved the following.
\begin{theorem}[Gerencs\'er and Gy\'arf\'as, \cite{Gerencser}]\label{GerencserGyarfas}
Every 2-edge-coloured complete graph can be covered by two disjoint monochromatic paths of different colours.
\end{theorem}

The proof of Theorem~\ref{GerencserGyarfas} is so short that it was originally published in a footnote of~\cite{Gerencser}.  Indeed to see that the theorem holds, simply find a red path $R$ in $K_n$ and a disjoint blue path $B$ in $K_n$ such that $|R|+|B|$ is as large as possible.  Let $r$ and $b$ be the endpoints of $R$ and $B$ respectively.  If there is a vertex $x \not\in R\cup B$, then it is easy to see that the triangle $\{x,r,b\}$ contains either a red path between $x$ and $r$ or a blue path between $x$ and $b$.  This path can be joined to $R$ or $B$ contradicting maximality of $|R|+|B|$.

Any result about vertex-partitioning coloured graphs into a small number of monochromatic subgraphs will imply a Ramsey-type result as a corollary.
For example Theorem~\ref{GerencserGyarfas} implies the bound $R(P_n, P_m)\leq n+m-1$.  Indeed Theorem~\ref{GerencserGyarfas} shows that every 2-edge-coloured $K_{n+m-1}$ can be covered by a red path $R$ and a disjoint blue path~$B$.  Clearly these paths cannot cover all the vertices unless $|R|\geq n$ or $|B|\geq m$.
This is the main technique we shall use to bound Ramsey numbers in this paper.

Although Theorem~\ref{GerencserGyarfas} originated as a technique to bound Ramsey Numbers, it subsequently gave birth to the area of partitioning edge-coloured complete graphs into monochromatic subgraphs.
There have been many further results and conjectures in this area, many of which generalise Theorem~\ref{GerencserGyarfas}. One particularly relevant conjecture which attempts to generalize Theorem~\ref{GerencserGyarfas} is the following.
\begin{conjecture} [Gy\'arf\'as,~\cite{Gyarfas}] \label{Gyarfas}  
The vertices of every $r$-edge-coloured complete graph can be covered with $r$ vertex-disjoint monochromatic paths.
\end{conjecture}

Although Theorems~\ref{PathRamsey} and~\ref{GerencserGyarfas} have both led to many generalizations, there have not been many further attempts to use results about partitioning coloured graphs in order to bound Ramsey Numbers.  A notable exception is the following result of Gy\'arf\'as and Lehel.
\begin{theorem}  [Gy\'{a}rf\'{a}s \& Lehel,~\cite{Gyarfas2, Lehel}]   \label{GyarfasLehel}
Suppose that the edges of $K_{n,n}$ are coloured with two colours such that one of the parts of $K_{n,n}$ is contained in a monochromatic connected component.   Then there exist two disjoint monochromatic paths with different colours which cover all, except possibly one, of the vertices of $K_{n,n}$.
\end{theorem}
Gy\'arf\'as and Lehel used this result to determine the bipartite Ramsey Number of a path i.e. the smallest $n$ for which every 2-edge-coloured $K_{n,n}$ contains a red copy of $P_i$ or a blue $P_j$.
Recently Theorem~\ref{GyarfasLehel} was used by the author in the proof of the $r=3$ case of Conjecture~\ref{Gyarfas} \cite{PokrovskiyCycles}.

In this paper we prove a new theorem about partitioning 2-edge-coloured complete graphs, and use it to determine certain Ramsey Numbers exactly.
Our starting point will be a lemma used by the author in the proof of the $r=3$ case of Conjecture~\ref{Gyarfas}.  

A complete bipartite graph is called \emph{balanced} if both of its parts have the same order.
The following lemma appears in~\cite{PokrovskiyCycles}.

\begin{lemma}\label{PathBipartite}
 Suppose that the edges of $K_n$ is coloured with two colours.  Then $K_n$ can be covered by a red path and a disjoint blue balanced complete bipartite graph.
\end{lemma}
Lemma~\ref{PathBipartite} immediately implies the bound $R(P_n,K_{m,m})\leq n+2m-2$.   It turns out that when $m \equiv 1 \pmod{n-1}$, this bound is best possible.  The following theorem was proved by H\"aggkvist.
\begin{theorem}[H\"aggkvist, \cite{Haggkvist}]\label{Haggkvist}
If $m,\ell \equiv 1 \pmod{n-1}$, then we have
$$R(P_n, K_{m,\ell})= n+m+\ell-2.$$ 
 \end{theorem}
The lower bound on Theorem~\ref{Haggkvist} comes from considering a colouring of $K_{n+m+\ell-3}$ consisting of $1+\ILF{(m+\ell-2)}{(n-1)}$ red copies of $K_{n-1}$ and all other edges are coloured blue.  The condition $m,\ell \equiv 1 \pmod{n-1}$ ensures that the number $1+\frac{(m+\ell-2)}{(n-1)}$ is an integer.

The main theorem about partitioning coloured graphs that we will prove in this paper is a generalization of Lemma~\ref{PathBipartite}.
Recall that a balaced complete $k$-partite graph, $\kPartite{m}{k}$, is a graph whose vertices can be partitioned into $k$ sets $A_1, \dots, A_k$ such that $|A_1|=\dots=|A_k|=m$ for all $i$, and there is an edge between $a_i\in A_i$ and $a_j\in A_j$ if, and only if, $i\neq j$.
We will prove the following.
\begin{theorem}\label{PathMultipartite}
 Let $k\geq 1$.  Suppose that the edges of $K_n$ are coloured with two colours.  Then $K_n$ can be covered by $k$ disjoint red paths and a disjoint blue balanced complete $(k+1)$-partite graph.
\end{theorem}

As a corollary of Theorem~\ref{PathMultipartite} we obtain that for all $m$ satisfying $m\equiv 1 \pmod{n-1}$ we have $R(P_n, \kPartite{m}{t})= (t-1)(n-1)+t(m-1)+1$.
This generalizes a result of Erd\H{o}s who showed that $R(P_n, K_m)=(t-1)(n-1)+1$ (see \cite{ErdosRamsey, Parsons}).

Instead of proving Theorem~\ref{PathMultipartite} directly, we will actually prove a strengthening of it, and then deduce Theorem~\ref{PathMultipartite} as a corollary.  The strengthening that we prove is the following.
\begin{theorem}\label{TreeMultipartite}
 Let $k\geq 1$.  Suppose that the edges of $K_n$ are coloured with the colours red and blue, such that the red spanning subgraph is connected.  Then $K_n$ can be covered by red tree $T$ with at most $k$ leaves and a disjoint blue balanced complete $(k+2)$-partite graph.
\end{theorem}

It is not immediately clear that Theorem~\ref{TreeMultipartite} implies Theorem~\ref{PathMultipartite}.  Notice that every tree with $k$ leaves can be covered by $k-1$ vertex-disjoint paths.  Therefore Theorem~\ref{TreeMultipartite} has the following corollary.

\begin{corollary}\label{PathMultipartiteConnected}
 Let $k\geq 1$.  Suppose that the edges of $K_n$ are coloured with the colours red and blue, such that the red spanning subgraph is connected.  Then $K_n$ can be covered by $k$ disjoint red paths and a disjoint blue balanced complete $(k+2)$-partite graph.
\end{corollary}
Corollary~\ref{PathMultipartiteConnected} shows that when the colouring of $K_n$ is connected in red, then the conclusion of Theorem~\ref{PathMultipartite} can actually be strengthened---we can use one less red path in the covering of $K_n$.  

Theorem~\ref{PathMultipartite} is easy to deduce from Corollary~\ref{PathMultipartiteConnected}.  
\begin{proof}[Proof of Theorem~\ref{PathMultipartite}]
For $k=1$, Theorem~\ref{PathMultipartite} is just Lemma~\ref{PathBipartite}.  This lemma was originally proven in \cite{PokrovskiyCycles}, and this proof is also reproduced in Section~\ref{SectionRamseyNotation}.  We shall therefore assume that $k\geq 2$.

Suppose that we have an arbitary 2-edge-colouring of $K_n$. We add an extra vertex $v$ to the graph and add red edges between $v$ and all other vertices.  The resulting colouring of $K_{n+1}$ is connected in red.  Therefore we can apply Corollary~\ref{PathMultipartiteConnected} to $K_n+v$ in order to cover it by $k-1$ disjoint red paths and a disjoint blue balanced $(k+1)$-partite graph $H$.  Since all the edges containing $v$ are red, the vertex $v$ cannot be in $H$.  Therefore, $v$ must be contained in one of the red paths.  Therefore, removing $v$ gives a partition of $K_n$ into $k$ disjoint red paths a blue balanced complete $(k+1)$-partite graph as required.
\end{proof}

A well known remark of Erd\H{o}s and Rado says that any 2-edge-coloured complete graph is connected in one of the colours.  Therefore Theorem~\ref{TreeMultipartite} implies that every 2-edge-coloured complete graph can be covered by a monochromatic path and a monochromatic balanced complete tripartite graph (where we have no control over which colour each graph has).

The $t$h power of a path of order $n$ is the graph constructed with vertex set $1, \dots, n$ and~$ij$ and edge whenever $1\leq |i-j|\leq t$.  It is easy to see that $\kPartite{m}{t}$ contains a copy of  $P_{tm}^{t-1}$.  Therefore, Theorem~\ref{PathMultipartite} and Corollary~\ref{PathMultipartiteConnected} imply the following.
\begin{corollary}\label{PathPathPower}
Let $k\geq 1$.  Suppose that $K_n$ is colored with two colours.
\begin{itemize}
 \item $K_n$ can be covered with $k$ disjoint red paths and a disjoint blue $k$th power of a path.
 \item If $K_n$ is connected in red, then $K_n$ can be covered with $k$ disjoint red paths and a disjoint blue $(k+1)$th power of a path.
\end{itemize}
\end{corollary}
The first part of this corollary may be seen as a generalization of Theorem~\ref{GerencserGyarfas}.
We are also able to use Corollary~\ref{PathPathPower} and Theorem~\ref{PathRamsey} to determine the  Ramsey numbers of a path on $n$ vertices versus a power of a path on $n$ vertices.
\begin{theorem}\label{PathPowerRamsey}
For all $k$ and $n\geq k+1$, we have
$$R(P_n, P_n^k)= (n-1)k + \left\lfloor \frac{n}{k+1} \right\rfloor.$$
\end{theorem}
Theorem~\ref{PathPowerRamsey} solves a conjecture of Allen, Brightwell, and Skokan who asked for the value of $R(P_n, P_n^k)$ in \cite{Skokan}.

The structure of this paper is as follows.  In Section \ref{SectionRamseyNotation} we define some notation and prove certain weakenings of Theorem~\ref{TreeMultipartite}.  
These weakenings serve to illustrate the main ideas used in the proof of Theorem~\ref{TreeMultipartite} and hopefully aid the reader in understanding that theorem.
In addition the results we prove in Section~\ref{SectionRamseyNotation} will be strong enough to imply Corollary~\ref{PathPathPower}.  This means that it is possible to prove Theorem~\ref{PathPowerRamsey} without using the full strength of Theorem~\ref{TreeMultipartite}.
In Section~\ref{SectionRamseyPartitioning} we prove Theorem~\ref{TreeMultipartite}.  
In Section~\ref{SectionRamseyNumbers} we prove Theorem~\ref{PathPowerRamsey} and also determine  $R(P_n, \kPartite{m}{t})$ whenever $m\equiv 1 \pmod{n-1}$.
In Section~\ref{SectionRamseyRemarks} we discuss some further problems which may be approachable using the techniques presented in this paper.
\
\section{Preliminaries}\label{SectionRamseyNotation}

For a nonempty path $P$, it will be convenient to distinguish between the two endpoints of $P$ saying that one endpoint is the ``start'' of $P$ and the other is the ``end'' of $P$.  Thus we will often say things like ``Let $P$ be a path from $u$ to $v$''.  Let $P$ be a path from $a$ to $b$ in $G$ and $Q$ a path from $c$ to $d$ in $G$.  If $P$ and $Q$ are disjoint and $bc$ is an edge in $G$, then we define $P+Q$ to be the unique path from $a$ to $d$ formed by joining $P$ and $Q$ with the edge $bc$.    If $P$ is a path and $Q$ is a subpath of $P$ sharing an endpoint with $P$, then $P-Q$ will denote the subpath of $P$ with vertex set $V(P)\setminus V(Q)$.

Whenever a graph $G$ is covered by vertex-disjoint subgraphs $H_1, H_2, \dots, H_k$, we say that $H_1, H_2, \dots, H_k$ \emph{partition} $G$.

All colourings in this section will be edge-colourings.
Whenever a graph is coloured with two colours, the colours will be called ``{red}" and ``{blue}". 
If a graph $G$ is coloured with some number of colours we define the \emph{{red} colour class} of $G$ to be the subgraph of $G$ with vertex set $V(G)$ and edge set consisting of all the {red} edges of $G$.  
We say that $G$ is \emph{connected in} {red}, if the {red} colour class is a connected graph.  Similar definitions are made for the colour {blue} as well.

For all other notation, we refer to \cite{Diestel}

In order to illustrate the main ideas of the proof of Theorem~\ref{TreeMultipartite}, we give a proof of Lemma~\ref{PathBipartite} here.

\begin{proof}[Proof of Lemma~\ref{PathBipartite}]
Notice that a graph with no edges is a complete bipartite graph (with one of the parts empty).  Therefore, any 2-edge-coloured $K_n$ certainly has a partition into a red path and a blue complete bipartite graph (by assigning all of $K_n$ to be one of the parts of the complete bipartite graph).
Partition $K_n$ into a red path $P$ and a complete bipartite graph $B(X,Y)$ with parts $X$ and $Y$ such that the following hold.
\begin{enumerate}[\normalfont (i)]
 \item $\max(|X|, |Y|)$ is as small as possible.
\item $|P|$ is as small as possible (whilst keeping (i) true).
\end{enumerate}
We are done if $|X|=|Y|$ holds.  Therefore, without loss of generality, suppose that we have $|X|<|Y|$.  

Suppose that $P=\emptyset$.  Then let $y$ be any vertex in $Y$, $P'=\{y\}$, $Y'=Y-y$, and $X'=X$.  This new partition of $K_n$ satisfies $\max(|Y'|,|X'|)<|Y|=\max(|X|,|Y|)$, contradicting minimality of the original partition in (i).

Now, suppose that $P$ is nonempty. 
 Let $p$ be an end vertex of~$P$.

If there is a red edge $py$ for $y\in Y$, then note that letting $P'=P+y$ and $Y'=Y-y$ gives a partition of $K_n$ into a red path and a complete bipartite graph~$B(X,Y')$ with parts $X$ and $Y'$.  However we have $\max(|Y'|,|X|)<|Y|=\max(|X|,|Y|)$, contradicting minimality of the original partition in (i).

If all the edges between $p$ and $Y$ are blue, then note that letting $P'=P-p$ and $X'=X+p$ gives a partition of $K_n$ into a red path and a complete bipartite graph $B(X',Y)$ with parts $X'$ and $Y$.  We have that $\max(|X'|,|Y|)=|Y|=\max(|X|,|Y|)$ and $|P'|<|P|$, contradicting minimality of the original partition in (ii).
\end{proof}

The proof of Theorem~\ref{TreeMultipartite} is similar to the above proof.
The above proof of Lemma~\ref{PathBipartite}  could be summarised as ``first we find a partition of our graph which is in some way extremal and then we show that it possesses the properties that we want''.  
The proof of Theorems~\ref{TreeMultipartite} has the same basic structure.

For a set $S\subseteq K_n$, let $c(S)$ be the order of the largest red component of $K_n[S]$.  
We now prove the following weakening of Theorem~\ref{TreeMultipartite}.
\begin{theorem}\label{WeakTreeMultipartite} 
 Let $k\geq 2$.  Suppose that the edges of $K_n$ are coloured with the colours red and blue, such that the red spanning subgraph is connected.  Then $K_n$ can be covered by a red tree with at most $k$ leaves and a disjoint set $S$ satisfying  $c(S)\leq |S|/(k+1)$.
\end{theorem}

Notice that Theorem~\ref{WeakTreeMultipartite} is indeed a weakening of Theorem~\ref{TreeMultipartite}.
To see this, simply note that if we have a set $S\subseteq V(K_n)$ such that the induced colouring of $K_n$ on $S$ is a blue balanced $(k+2)$ partite graph, then $S$ satisfies $c(S)\leq |S|/(k+2)$.

\begin{proof}[Proof of Theorem~\ref{WeakTreeMultipartite}]
We partition $K_n$ into a red tree $T$ and a set $S$ with the following properties.
\begin{enumerate}[\normalfont (i)]
\item $T$ has at most $k$ leaves.
 \item $c(S)$ is as small as possible (whilst keeping (i) true).
\item The number of red components in $S$ of order $c(S)$ is as small as possible (whilst keeping (i) and (ii) true).
\item $|T|$ is as small as possible (whilst keeping (i) -- (iii) true).
\end{enumerate}
We claim that $c(S)\leq |S|/(k+1)$ holds.  Suppose otherwise that we have $c(S)>|S|/(k+1)$.  Notice that since c(S) is an integer, this implies $c(S)\geq \lfloor|S|/(k+1)\rfloor+1$.
We will construct a new partition of $K_n$ into a tree $T'$ and a set $S'$ which will contradict minimality of the original partition in either (ii), (iii), or (iv).

Let $S^+$ be subset of $S$ formed by taking the union of the red components of order $c(S)$ in $S$.  Let $S^-$ be $S\setminus S^+$.  Since we are assuming $c(S)>|S|/(k+1)$, we must have $|S^-|<k|S|/(k+1)$.

Let $v_1, \dots, v_{\ell}$ be the leaves of $T$.  By assumption (i), we have $\ell\leq k$.

Suppose that $v_i$ has a red neighbour $u\in S^+$.  Then we can let $T'=T+u$ be the red tree formed from $T$ by adding the edge $v_iu$, and $S'=V(K_n)\setminus V(T)$.  Notice that $T'$ still has at most $k$ leaves.  Since $S'$ is a subset of $S$, we must have $c(S')=c(S)$ (by minimality of $c(S)$ in (ii)).  But since $u$ was in a red component of order $c(S)$, $S'$ must have one less component of order $c(S)$ than $S$ had.  This contradicts minimality of the original partition in (iii).

For the remainder of the proof, we can suppose that the vertices $v_1, \dots, v_{\ell}$ do not have any red neighbours in $S^+$.  For a leaf $v_i$, let $\overline{N}(v_i)$ be red connected component containing $v_i$ in the induced graph on $S^-+v_i$

Suppose that $\overline{N}(v_i)\cap \overline{N}(v_i)\neq \emptyset$ for some $i\neq j$.  Then there must be a red path $P$ between $v_i$ and $v_j$ contained in $S^-+v_i+v_j$.  Let $T_1$ be the graph formed by adding the path $P$ to the tree $T$.  Notice that $T_1$ is a red graph with $\ell-2$ leaves and exactly one cycle.  By connectedness of the red colour class of $K_n$ there is a red edge between some $v \in T_1$ and $u\in S^+$.  Let $T_2$ be the graph formed by adding the vertex $u$ and the edge $uv$ to $T_1$.  Notice that $T_2$ is a red graph with between $1$ and $\ell-1$ leaves and exactly one cycle.  Therefore $T_2$ contains an edge  $xy$ which is contained on the cycle and the vertex $x$ has degree at least $3$.  Let $T_3$ be $T_2$ minus the edge $xy$ and $S'=V(K_n)\setminus V(T_2)$.  Now $T_3$ is a red tree with at most $\ell\leq k$ leaves.  As before $S'\subset S$ and (ii) implies that we must have $c(S')=c(S)$.  As before this contradicts (iii) since the vertex $u$ which we removed from $S^+$ was contained in a red component of order $c(S)$.

Suppose that $\overline{N}(v_i)\cap \overline{N}(v_i)= \emptyset$ for all $i\neq j$.  Recall that we have $\overline{N}(v_i)-vi\subseteq S^-$ for all $i$ and also $|S^-|<k|S|/(k+1)$.  By the Pigeonhole Principle, for some $i$ we have $|\overline{N}(v_i)-v_i|< |S|/(k+1)$, which combined with the integrality of $|\overline{N}(v_i)|$ implies $|\overline{N}(v_i)| \leq \lfloor|S|/(k+1)\rfloor+1$.  Let $T'=T-v_i$ and $S'=V(K_n)\setminus T'$.  The new $T'$ satisfies (i).  The only red component of $S'$ which was not a red component of $S$ is $\overline{N}(v_i)$.  However we have $\overline{N}(v_i)\leq \lfloor|S|/(k+1)\rfloor+1\leq c(S)=c(S')$ and so $S'$ satisfies (ii) and (iii).  However we have $|T'|=|T|-1$ contradicting minimality of $|T|$ in (iv).

This completes the proof of the theorem.
\end{proof}

Most of the steps of the above proof reoccur in the proof of Theorem~\ref{TreeMultipartite}.
We conclude this section by observing that Theorem~\ref{WeakTreeMultipartite} implies Corollary~\ref{PathPathPower}.  

First notice that it is sufficient to prove the following proposition.
\begin{proposition}\label{WeakSeymour}
Let $K_n$ be a 2-edge-coloured complete graph.  Suppose that $K_n$ contains a set $S$ which satisfies  $c(S)\leq |S|/(k+2)$.  Then $S$ contains a spanning blue $(k+1)$st power of a path.
\end{proposition}

Indeed combining Proposition~\ref{WeakSeymour} with Theorem~\ref{WeakTreeMultipartite} we obtain that every 2-edge-coloured complete graph which is connected in red can be covered by a red tree $T$ with at most $k$ leaves and a spanning blue $(k+1)$st power of a path.
Since every tree with at most $k$ leaves can be partitioned into $k-1$ disjoint paths, this implies part (ii) of Corollary~\ref{PathPathPower}.  For $k\geq 2$,  part (i) of Corollary~\ref{PathPathPower} follows from part (ii) in exactly the same way as we deduced Theorem~\ref{PathMultipartite} from Corollary~\ref{PathMultipartiteConnected} in the introduction.  Indeed, to prove part (i) of Corollary, we start with an arbitary colouring of $K_n$.  We add a vertex $v$ to the graph and add red edges between $v$ and all other vertices.  The resulting colouring of $K_{n+1}$ is connected in red.  Therefore we can apply part (ii) of Corollary~\ref{PathPathPower} to $K_n+v$ in order to cover it by $k-1$ disjoint red paths and a disjoint blue $k$th power of a path $P$.  Since all the edges containing $v$ are red, the vertex $v$ cannot be in $P$ (unless $|P|\leq 1$).  Therefore, removing $v$ gives a partition of $K_n$ into $k$ disjoint red paths a blue $k$th power of a path as required.

It remains to verify Proposition~\ref{WeakSeymour}.
One way of doing this is to notice that if  $c(S)\leq |S|/(k+2)$, then the induced blue subgraph of $K_{n}$ on $S$ must have minimal degree at least $\frac{k+1}{k+2}|S|$.  A conjecture of Seymour says that all graphs with minimal degree $\frac{k}{k+1}|S|$ contain a $k$th power of Hamiltonian cycle \cite{SeymourConjecture}.   Seymour's Conjecture has been proven for graphs with sufficiently large order by Koml\'os, S\'ark\"ozy, and Szemer\'edi \cite{Komlos}.  Seymour's Conjecture readily implies Proposition~\ref{WeakSeymour}.  However given that our set $S$ has a very specific structure, it is not hard to prove that it contains a spanning blue $(k+1)$st power of a path without using the full strength of Koml\'os, S\'ark\"ozy, and Szemer\'edi's result.  One  way of doing this is by induction on the number of vertices of $S$.  We omit the details,  because Corollary~\ref{PathPathPower} follows much more readily from the stronger Theorem~\ref{TreeMultipartite}.

\section{Partitioning coloured complete graphs}\label{SectionRamseyPartitioning}
In this section we prove Theorem~\ref{TreeMultipartite}.
We first prove an intermediate lemma.  The following lemma will allow us to take a partition of $K_n$ into a red tree $T$ and a blue multipartite graph $H$ which is ``reasonably balanced,'' and output a partition of $K_n$ into a red tree and a blue balanced complete $(k+1)$-partite graph that we require.

\begin{lemma} \label{Balancedlemma}
Suppose that we have a $2$-edge-coloured complete graph $K_n$  containing $k+1$ sets $A_0, \dots A_k$, $k$ sets $B_1, \dots B_k$, and $k$ sets $N_1, \dots, N_k$ such that the following hold.
\begin{enumerate}[\normalfont (i)]
\item \label{Con:Partition} The sets $A_0, \dots A_k, B_1, \dots B_k$ partition $V(K_n)$.
\item \label{Con:Blue} For all $1\leq i < j\leq k$ all the edges between any of the sets $A_0$, $A_i$, $B_i$, $A_j$, and $B_j$ are blue.
\item \label{Con:N} For all $i$, every red component of $B_i$ intersects $N_i$.
\item \label{Con:A} $|A_0| \geq |A_i|$ for all $i\geq 1$.
\item \label{Con:AB} $|A_i|+|B_i|\geq |A_0|$ for all $i\geq 1$.

\item \label{Con:Last} For all $i\geq 1$ either $|B_i|\leq 2 \min_{t=1}^k |B_t|$ or $|A_i| + |B_i|\leq |A_0| + \min_{t=1}^k |B_t|$ holds.
\end{enumerate}
Then, there is a partition of $K_n$ into $k$ red paths $P_1, \dots, P_k$ and a blue balanced $k+1$ partite graph.  In addition, for each $i$, the path $P_i$ is either empty or starts in $N_i$.
\end{lemma}
\begin{proof}
The proof is by induction on the quantity $\sum_{t=1}^k |B_t|$.  

First we prove the base case of the induction, i.e. we prove the lemma when $\sum_{t=1}^k |B_t|=0$.  In this case $B_i=\emptyset$ for all $i$, and so conditions (\ref{Con:A}) and (\ref{Con:AB}) imply that $|A_i|=|A_0|$ for all $i$.  Therefore, by (ii), $K_n$ contains a spanning blue complete $(k+1)$-partite graph with parts $A_0, \dots, A_k$. We can take $P_1=\dots= P_k=\emptyset$ to obtain the required partition.

We now prove the induction step.
Suppose that the lemma holds for all  $2$-edge-coloured complete graphs $K_n'$ containing sets $A'_0, \dots A'_k$, $B'_1, \dots B'_k$, and $N'_1, \dots, N'_k$ as in the statement of the lemma but satisfying $\sum_{t=1}^k |B'_t|< \sum_{t=1}^k |B_t|$.  We will show that the lemma holds for $K_n$ as well.

First we show that if there is a partition of $K_n$ satisfying (\ref{Con:Partition}) -- (\ref{Con:Last}), then the sets $A_0,\dots, A_k$ and $B_1,\dots, B_k$ can be relabeled to obtain a partition satisfying (\ref{Con:Partition}) -- (\ref{Con:Last})  and also the following
\begin{align}
|A_0|\geq |A_1| &\geq \dots \geq |A_k| \label{Amonotone},\\
|B_1| &\leq \dots \leq |B_k| \label{Bmonotone}.
\end{align}
  The following claim guarantees this.

\begin{claim}
Let $\sigma$ be a permutation of $(0,1, \dots, k)$ ensuring that $|A_{\sigma(0)}|\geq |A_{\sigma(1)}| \geq \dots \geq |A_{\sigma(k)}|$ holds.  Let $\tau$ be a permutation of $(1, \dots, k)$ ensuring that $|B_{\tau(1)}| \leq \dots \leq |B_{\tau(k)}|$ holds.  
Let $A'_i=A_{\sigma(i)}$, $B'_i=B_{\tau(i)}$, and $N'_i=N_{\tau(i)}$.
Then the sets $A'_i$, $B'_i$, and $N'_i$  satisfy (\ref{Con:Partition}) -- (\ref{Con:Last}).
\end{claim}

\begin{proof}
Notice that $P'_i$, $A'_i$ and $B'_i$ satisfy (\ref{Con:Partition}) -- (\ref{Con:N}) trivially.

Since the sets $A_{i}$ satisfy (\ref{Con:A}), we can assume that $\sigma(0)=0$. This ensures that the sets $A_{\sigma(i)}$ satisfy (\ref{Con:A}).

For (\ref{Con:AB}), note that if for some $j\geq 1$,  $|A_{\sigma(j)}| + |B_{\tau(j)}|<|A_0|$, then we also have $|A_{\sigma(x)}|+|B_{\tau(y)}|<|A_0|$ for all $x\geq j$ and $y \leq j$.  However, the Pigeonhole Principle implies that $\sigma(x)=\tau(y)$ for some $x\geq j$ and $y \leq j$, contradicting the fact that $A_i$ and $B_i$ satisfy~(\ref{Con:AB}) for all $i$.

Suppose that (\ref{Con:Last}) fails to hold.  Then for some $j$, $|B_{\tau(j)}|> 2 \min_{t=1}^k |B_t|$ and $|A_{\sigma(j)}|+|B_{\tau(j)}|> |A_0| + \min_{t=1}^k |B_t|$ both hold.  
If we have $|A_{\tau(i)}|\geq |A_{\sigma(j)}|$  for some $i\geq j$, then $|B_{\tau(i)}|\geq |B_{\tau(j)}|> 2 \min_{t=1}^k |B_t|$ and $|A_{\tau(i)}|+|B_{\tau(i)}|\geq |A_{\sigma(j)}|+|B_{\tau(j)}| > |A_0| + \min_{t=1}^k |B_t|$ both hold, contradicting the fact that $A_i$ and $B_i$ satisfy (\ref{Con:Last}) for all $i$.
Therefore, we can assume that $|A_{\tau(i)}|< |A_{\sigma(j)}|$ for all $i\geq j$.  This, together with  $|A_{\sigma(0)}|\geq |A_{\sigma(1)}| \geq \dots \geq |A_{\sigma(k)}|$ implies that $\{\tau(j), \tau(j+1), \dots, \tau(k)\}\subseteq \{\sigma(j+1), \sigma(j+2), \dots, \sigma(k)\}$, contradicting  $\tau$ being injective.
\end{proof}

By the above claim, without loss of generality we may assume that the  $A_i$s and $B_i$s satisfy (\ref{Amonotone}) and (\ref{Bmonotone}). 

Notice that the lemma holds trivially if we have the following.
\begin{equation}\label{classesequal}
|A_0|=|A_1|+|B_1|= |A_2|+|B_2|=\dots=|A_k|+|B_k|.
\end{equation}
Indeed, if (\ref{classesequal}) holds, then $K_n$ contains a spanning blue complete $(k+1)$-partite graph with parts $A_0,A_1\cup B_1 \dots, A_k\cup B_k$, and so taking $P_1=\dots= 
P_k=\emptyset$ gives the required partition.

 Therefore, we can assume that (\ref{classesequal}) fails to hold, so there is some $j$  such that $|A_j|+|B_j|>|A_0|$. In addition, we can assume that $j$ is as large as possible, and so $|A_i|+|B_i|=|A_0|$ for all $i> j$.  
 
 First we deal with the case when $|B_j|\leq 1$.  Notice that in this case (\ref{Bmonotone}) implies that $|B_i|\leq 1$ for all $i\leq j$.  Therefore for each $i$ satisfying $|A_i|+|B_i|>|A_0|$, we have $|B_i|=1$ and we can let $P_i$ be the single vertex in $B_i$.  For all other $i$, we let $P_i=\emptyset$.
 This ensures that $K_n\setminus(P_1, \dots, P_k)$ is a balanced complete $k$-partite graph with classes $A_1, \dots, A_j, A_{j+1}\cup B_{j+1}, \dots, A_k\cup B_k$, giving the required partition of $K_n$.
 
 For the remainder of the proof, we assume that $|B_j|\geq 2$.
 We split into two cases depending on whether $B_j$ is connected in red or not.

\textbf{Case 1:}  Suppose that $B_j$ is connected in red.  Let $v$ be a vertex in $B_j\cap N_j$.  Let $K'_n=K_n-v$, $B'_j= B_j - v$, $N'_j=N_r(v)$ and $A'_i=A_i, B'_i=B_i, N'_i=N_i$ for all other $i$.  We show that the graph $K'_n$ with the sets $A'_i$, $B'_i$, and $N'_i$ satisfies (\ref{Con:Partition}) -- (\ref{Con:Last}). 

Conditions (\ref{Con:Partition}), (\ref{Con:Blue}), and (\ref{Con:A}) hold trivially for the new sets  as a consequence of them holding for the original sets $A_j$ and $B_j$.
Condition (\ref{Con:N}) holds trivially whenever $i\neq j$, and holds for $i=j$ as a consequence of $B_j$ being connected in red.

To prove (\ref{Con:AB}), it is sufficient to show that  $|A'_j|+|B' _j|\geq |A'_0|$.  This is equivalent to $|A_j|+|B_j-v|\geq |A_0|$, which holds as a consequence of  $|A_j|+|B_j|>|A_0|$. 

We now prove (\ref{Con:Last}).
Note that we have $\min_{t=1} ^k |B' _t|=\min(|B_1|, |B' _j|)$.  If $\min_{t=1} ^k |B' _t|=|B_1|$ holds, then (vi) is satisfied for the new sets $A'_0,\dots,A'_k, B'_0, \dots, B'_k$  as a consequence of it being satisfied for the original sets $A_0,\dots,A_k, B_0, \dots, B_k$.  Now, suppose that we have $\min_{t=1} ^k |B' _t|=|B' _j|$.  For $i>j$, we have $|A'_i| + |B'_i|=|A'_0|$ which implies that (vi) holds for these $i$.  If $i\leq j$, then we have $|B_i|\leq |B_j|$ which together with $|B_j|\geq 2$ implies that $B'_i\leq 2|B_j|-2=2|B'_j|$ holds. 

Therefore, the graph $K'_n$ with the sets $A'_i$, $B'_i$,and $N'_i$ satisfies (\ref{Con:Partition}) -- (\ref{Con:Last}).  We also have $\sum_{t=1}^k |B'_t|= \sum_{t=1}^k |B_t|-1$, and so, by induction  $K'_n$ can be partitioned into $k$ red paths $P'_1, \dots, P'_k$ starting in $N'_1, \dots, N'_k$ respectively and a blue balanced $k+1$ partite graph $H$.  Since $P'_j$ starts in $N'_j=N_r(v)$, we have the required partition of $K_n$ into $k$ paths $P'_1, \dots,v+P'_j,\dots  P'_k$ and a blue balanced $k+1$ partite graph $H$.

\textbf{Case 2:}  Suppose that $B_j$ is disconnected in red.  
We will find a new partition of $K_n$ into sets $A'_0, \dots, A'_k$ and $B'_1, \dots, B'_k$, which together with $N_1, \dots, N_k$ satisfy (\ref{Con:Partition}) -- (\ref{Con:Last}).  We will also have  $\sum_{t=1}^k |B'_t|< \sum_{t=1}^k |B_t|$ which implies the lemma by induction.

Let $B^- _j$ be the smallest red component of $B_j$ and $B^+ _j= B_j\setminus B^- _j$.  
There are two subcases, depending on whether we have $|A_j|+|B^-_j|\leq |A_0|$ or not.

\textbf{Case 2.1:} Suppose that we have $|A_j|+|B^-_j|\leq |A_0|$.  
Let $B'_j= B^+ _j$ and $A'_j= A_j\cup B^-_j$, and $A'_i=A_i$, $B'_i=B_i$ for all other $i$.
As before, conditions (\ref{Con:Partition}) -- (\ref{Con:N}) hold trivially.

To prove (\ref{Con:A}), it is sufficient to show that $|A'_0|\geq |A'_j|$ which is true since we are assuming that $|A_j|+|B^-_j|\leq |A_0|$.  

To prove (\ref{Con:AB}), it is sufficient to show that $|A'_j|+|B'_j|\geq |A'_0|$ which holds since  we have $|A'_j|+|B'_j|= |A_j|+|B_j|\geq |A_0|$.

To prove (\ref{Con:Last}), note that we have $\min_{t=1}^k |B'_t|=\min(B_1, B'_j)$.  If $\min_{t=1} ^k |B' _t|=|B_1|$ holds, then (vi) is satisfied for the new sets $A'_0,\dots,A'_k, B'_0, \dots, B'_k$  as a consequence of it being satisfied for the original sets $A_0,\dots,A_k, B_0, \dots, B_k$.  Now, suppose that we have $\min_{t=1} ^k |B' _t|=|B' _j|$.  For $i>j$, we have $|A'_i| + |B'_i|=|A'_0|$ which implies that (vi) holds for these $i$.  If $i\leq j$, then we have $|B_i|\leq |B_j|$ which together with $|B_j|\leq 2|B^+_j|$ implies that $B'_i\leq 2|B'_j|$ holds.

Notice that we have $\sum_{t=1}^k |B'_t|< \sum_{t=1}^k |B_t|$, and so the lemma holds by induction.

\textbf{Case 2.2:}
 Suppose that we have $|A_j|+|B^-_j|> |A_0|$. 
 
We claim that in this case $|B_j|\leq 2\min_{t=1}^k |B_t|$ holds.  Indeed by (\ref{Con:Last}), we have that either $|B_j|\leq 2\min_{t=1}^k |B_t|$ holds, or we have $|A_j|+|B_j|\leq |A_0| + \min_{t=1}^k |B_t|$.  Adding  $|A_j|+|B_j|\leq |A_0| + \min_{t=1}^k |B_t|$ to $|A_j|+|B^-_j|> |A_0|$ gives $|B^+_j|< \min_{t=1}^k |B_t|$.  This, together with  $|B_j|\leq 2|B^+_j|$ implies that  $|B_j|\leq 2\min_{t=1}^k |B_t|$ always holds.
  
There are two cases, depending on whether we have $j=k$ or not.

Suppose that $j\neq k$.  Let $B'_j= B^+ _j$, $A'_{j+1}=A_{j+1}\cup B^-_j$, and $A'_i=A_i$, $B'_i=B_i$ for all other $i$.
As before, conditions (\ref{Con:Partition}) -- (\ref{Con:N}) hold trivially.

To prove (\ref{Con:A}), it is sufficient to show that $|A'_0|\geq |A'_{j+1}|$, which holds as a consequence of $|A_{j+1}|+|B_{j+1}|=|A_0|$ and (\ref{Bmonotone}).

To prove (\ref{Con:AB}), it is sufficient show that $|A'_j|+|B'_j|\geq |A'_0|$, which holds as a consequence of  $|B^+_j|\geq |B^-_j|$ and $|A_j|+|B^-_j|> |A_0|$.

We now prove (\ref{Con:Last}).  For $i\geq j+2$, note that we have $|A'_i|+|B'_i|=|A'_0|$ which implies that (vi) holds for these $i$.  
For  $i\leq j$, (vi) holds since we have $|B'_i|\leq |B_j|\leq 2|B^+_j|=~|B'_j|$.  For $i=j+1$, we have $|A'_{j+1}|+|B'_{j+1}|\leq |A'_0|+ \min_{t=1}^k|B'_t|$  as a consequence of  $|A'_{j+1}|+|B'_{j+1}|= |A_0|+ |B^-_j|$, $|B^- _j|\leq \frac{1}{2}|B_j|$, and $|B_j|\leq 2\min_{t=1}^k|B_t|$.

Notice that we have $\sum_{t=1}^k |B'_t|< \sum_{t=1}^k |B_t|$, and so the lemma holds by induction.

Suppose that $j= k$.
Let $B'_k= B^+ _k$, $A'_k= A_0$, $A'_{0}=A_{k}\cup B^-_k$, and $A'_i=A_i, B'_i=B_i$ for all other $i$.
As before, conditions (\ref{Con:Partition}) -- (\ref{Con:N}) hold trivially.

Since $|A_0|\geq |A'_i|$ for all $i\geq 1$, to prove (\ref{Con:A}), it is sufficient to show that $|A'_0|\geq |A_0|$. This holds since we assumed that $|A_k|+|B^-_k|> |A_0|$.

To prove (\ref{Con:AB}), we have to show that $|A_i|+|B_i|\geq |A_k|+|B^-_k|$ for all $i<k$ and also that $|A_0|+|B^+_k|\geq |A_k|+|B^-_k|$.
 We know that for all $i$ we have $|B^-_k|\leq \frac{1}{2}|B_k|\leq |B_i|$ which, combined with (\ref{Amonotone}), implies that we have $|A_i|+|B_i|\geq |A_k|+|B^-_k|$. We also know that $|B^+_k|\geq |B^-_k|$ which, combined with (\ref{Amonotone}), implies that we have $|A_0|+|B^+_k|\geq |A_k|+|B^-_k|$.

 To prove (\ref{Con:Last}), note that we have $\min_{t=1}^k |B'_t|=\min(B_1, B'_k)$.  
If $\min_{t=1} ^k |B' _t|=|B' _k|$ holds,  then we have $|B'_i|\leq 2|B'_k|$  for all $i$ as a consequence of  (\ref{Bmonotone}) and  $2|B'_k|\geq |B_k|$.
Suppose that $\min_{t=1} ^k |B' _t|=|B'_1|$ holds. Then for $i<k$, (vi) is satisfied for the new sets $A'_0,\dots,A'_k, B'_0, \dots, B'_k$  as a consequence of it being satisfied for the original sets $A_0,\dots,A_k, B_0, \dots, B_k$ and $|A'_0|\geq |A_0|$.  For $i=k$, (vi) holds since we have $|B'_k|\leq |B_k|\leq 2\min_{t=1}^k|B_t|$.

Notice that we have $\sum_{t=1}^k |B'_t|< \sum_{t=1}^k |B_t|$, and so the lemma holds by induction.
\end{proof}

We now use the above lemma to prove Theorem~\ref{TreeMultipartite}.  The proof has many similarities to that of Theorem~\ref{WeakTreeMultipartite}

\begin{proof}[Proof of Theorem~\ref{TreeMultipartite}.]

We will partition $K_n$ into a red tree $T$, and sets $A_0, A_1, \dots, A_k$ and $B_1, \dots, B_k$ with certain properties. For convenience we will define $A=A_0\cup A_1\cup \dots \cup A_k$   and $B=B_1\cup \dots \cup B_k$.  The tree $T$ will have $l$ leaves which will be called $v_1, v_2, \dots, v_l$.  
For a set $S\subseteq K_n$, let $c(S)$ be the order of the largest red component of $K_n[S]$.
Define $f(S)$ to be the number of red components contained in $S$ of order $c(A\cup B)$.  The tree $T$, and sets $A_0, A_1, \dots, A_k$ and $B_1, \dots, B_k$ are chosen to satisfy the following.

\begin{enumerate}[(I)]
\item For $1\leq i< j\leq k$, all the edges between $A_0$,$A_i$, $A_j$, $B_i$, and $B_j$ are blue.
\item $T$ has  $l$ leaves $v_1, \dots, v_l$, where $l\leq k$.  For $i=1, \dots, l$, the leaf  $v_i$, is joined to every red component of $B_i$ by a red edge.  
\item $c(A\cup B)$ is as small as possible, whilst keeping (I) -- (II) true.
\item $\sum_{t=1}^k|f(B_t)-\frac{1}{2}|$ is as small as possible, whilst keeping (I) -- (III) true.
\item $f(A)$ is as small as possible, whilst keeping (I) -- (IV) true.
\item $|T|$  is as small as possible, whilst keeping (I) -- (V) true.
\item $|\{i\in\{1, \dots, k\}: |B_i|\geq c(A\cup B)\}|$ is as large as possible, whilst keeping (I)~--~(VI) true.
\item $\sum_{\{t:|B_t|< c(A\cup B)\}} |B_t|$ is as large as possible, whilst keeping (I) -- (VII) true.
\item $\sum_{t=1}^k |B_t|$ is as small as possible, whilst keeping (I) -- (VIII) true.
\item $\max_{t=1}^k{|A_t|}$ is as small as possible, whilst keeping (I) -- (IX) true.
\item $|\{i\in\{1, \dots, k\}: |A_i|= \max_{t=1} ^k{|A_t|}\}|$ is as small as possible, whilst keeping (I)~--~(X) true.
\end{enumerate}

In order to prove Theorem~\ref{TreeMultipartite} we will show that the partition of $A\cup B$ into $A_i$ and~$B_i$ satisfies conditions (\ref{Con:Partition}), (\ref{Con:Blue}), (\ref{Con:A}), (\ref{Con:AB}), and  (\ref{Con:Last}) of Lemma~\ref{Balancedlemma}.  Then, Lemma~\ref{Balancedlemma} will easily imply the theorem.

Without loss of generality, we may assume that the $A_i$s are labelled such that we have
\begin{equation}\label{Aordered}
|A_0|\geq |A_1|\geq \dots \geq |A_k|.
\end{equation}

We begin by proving a sequence of claims.
\begin{claim}\label{onecomponent}
 For each $i$, $f(B_i)$ is either $0$ or $1$.
\end{claim}
\begin{proof}
 Suppose that $f(B_i)\geq 2$.  Let $C$ be a red component in $B_i$ of order $c(A\cup B)$.  Let $B'_i=B_i\setminus C$, $A'_0=A_0\cup C$, $T'=T$ and $A'_j=A_j$, $B'_j=B_j$ for other $j$.  It is easy to see that the new partition satisfies (I) -- (III).   
   We have that $f(B'_i)=f(B_i)-1$, which combined with $f(B_i)\geq 2$ implies that $|f(B'_i)-\frac{1}{2}|<|f(B_i)-\frac{1}{2}|$ contradicting minimality of the original partition in (IV).
\end{proof}

\begin{claim}\label{largeB}
 If we have $f(B_i)=1$ for some $i$, then we also have $|B_i|=c(A\cup B)$.
\end{claim}
\begin{proof}
Suppose that $f(B_i)=1$ and $|B_i|>c(A\cup B)$ both hold. Then $B_i$ contains some red connected component $C$ of order strictly less than $c(A\cup B)$.
Let $T'=T$, $A'_0=A\cup C$, $B'_i=B_i\setminus C$, and  $A'_t=\emptyset$, $B'_t=B_t$ for all other $t$.

It is easy to see that the new partition satisfies (I) -- (VIII).  However $|B'_i|<|B_i|$ and  $|B'_t|=|B_t|$ for $t\neq i$ contradicts minimality of the original partition in (IX).
\end{proof}

\begin{claim}\label{onelargecomponentinA}
 We have that $f(A)\geq 1$.
\end{claim}
\begin{proof}
Suppose that we have $f(A)=0$. Then all the red components of order $c(A\cup B)$ of $A\cup B$ must be contained in~$B$.  
For each $i\in\{1, \dots, k\}$, let $C_i$ be a red component of order $c(A\cup B)$ contained in $B_i$ (if one exists).  By Claim~\ref{onecomponent} any red component of $A\cup B$ or order $c(A\cup B)$ must be one of the~$C_i$s.  
By (II), for $i\in\{1, \dots, l\}$, if $C_i$ exists, then $v_i$ has a red neighbour $u_i$ in $C_i$.  By red-connectedness of $K_n$ and part (I), every $C_i$ must be connected to $T$ by a red edge.  Therefore, for $i\in\{l+1, \dots, k\}$, if $C_i$ exists, then there is a red edge $u_i w_i$ between $u_i\in C_i$ and some $w_i\in T$.  

Let $A'_0=A\cup B\setminus\{u_1, \dots, u_k\}$ and $A'_j=B'_j=\emptyset$ for $j\geq 1$.  Let $T'$ be the tree with vertex set $V(T)\cup\{u_1, \dots,u_k\}$ formed from $T$ by joining $u_i$ to $v_i$ for $i=1, \dots, l$ and $u_i$ to~$w_i$ for $i=l+1, \dots, k$.

Clearly the new partition satisfies (I) and (II).
However since each of the largest components of $A\cup B$ lost a vertex, we must have $c(K_n\setminus T)<c(A\cup B)$ contradicting minimality of the original partition in (III).
\end{proof}

\begin{claim}\label{unconnectedB}
 If $i> l$, then $f(B_i)=1$ holds.
\end{claim}

\begin{proof}
Suppose that $f(B_i)= 0$ for some $i$.

By Claim~\ref{onelargecomponentinA}, there is a red component $C$ of order $c(A\cup B)$ in $A$.
Let $T'=T$, $A'_0=A\setminus C$, $B'_i=B_i\cup C$, and  $A'_t=\emptyset$, $B'_t=B_t$ for all other $t$.

It is easy to see that the new partition satisfies (I) -- (IV).  However we have $f(A)=f(A)-1$ contradicting minimality of the original partition in (V).
\end{proof}

\begin{claim}\label{Abound}
For every $i$, we have $|A_0|\leq |A_i| + c(A\cup B)$.
\end{claim}
\begin{proof}
Suppose that for some $i$ we have $|A_0|> |A_i| + c(A\cup B)$.  Let $C$ be any red component of $A_0$.  We have $|C|\leq c(A\cup B)$.  Let $A'_0=A_0\setminus C$, $A'_i=A_i\cup C$, $T'=T$ and $A'_j=A_j$, $B'_j=B_j$ otherwise.  It is easy to see that $T'$, $A'_j$, and $B'_j$ will satisfy (I) -- (IX).  If the new partition satisfies (X), then we must have $\max_{t=0}^k |A'_t|=|A_0|$.
However $|A_0|> |A_i| + c(A\cup B)$ ensures that we have $|A'_0|, |A'_i|<|A_0|)$ meaning that the quantity $|\{i\in\{1, \dots, k\}: |A'_i|= |A'_0|\}|$ must be smaller than it was in the original partition, contradicting (XI). 
\end{proof}

\begin{claim} \label{BLowerBound}
For every $i$, we have $|B_i|\geq c(A\cup B)$.
\end{claim}

\begin{proof} 

Suppose that $|B_i|< c(A\cup B)$ for some $i$.  Notice that this implies that $f(B_i)=0$.
By Claim~\ref{unconnectedB}, we have that $i\leq l$.

First suppose that we have $N_r(v_i)\cap A\neq \emptyset$.  
Let $C$ be a red component of $A$ which intersects $N_r(v_i)$. 
Let  $T'=T$, $B'_i=B_i\cup C$, and $A'_t=A_t\setminus C$, $B'_t=B_t$, for other $t$.

The new partition satisfies (I) trivially.  By choice of $C$, new partition satisfies (II). 
It is easy to see that $c(A'_t), c(B'_t)\leq c(A\cup B)$ for every $t$ which implies that (III) holds for the new partition.
Since $f(B_i)=0$ holds, we have that $f(B'_i)\leq 1$ and hence $|f(B'_i)-\frac{1}{2}|=|f(B_i)-\frac{1}{2}|$ which implies that (IV) holds for the new partition.

It is easy to see that $f(A'_t)\leq f(A_t)$ for all $t$, which implies that (V) holds for the new partition.
Since $T'=T$, (VI) holds for the new partition.
 
We have that $|B'_t|\geq |B_t|$ for all $t$.  This implies that if the new partition satisfies~(VII), then we have $|B'_i|<c(A\cup B)$.  However since $|B'_i|>|B_i|$, this contradicts maximality of the original partition in (VIII)

For the remainder of the proof of this claim, we may assume that we have $N_r(v_i)\subseteq B$.
There are two cases depending on where the neighbours of $v_i$ lie.

\textbf{Case 1:}  Suppose that $N_r(v_i)\subseteq B_i$.

Let $T'=T-v_i$, $B'_i= B_i+v_i$, and $A'_j=A_j$, $B'_j=B_j$ for other $j$.  
The resulting partition  satisfies (I) since $N_r(v_i)\subseteq B_i$.
Condition (II) implies that $B_i+v_i$ is connected in red. This, together with the fact that the neighbour of $v_i$ in $T$ is connected to $B'_i$ by a red edge implies that  condition (II) holds for the new partition.
The only red component  of the new partition which was not a red component of the old partition is $B_i\cup v$, which is of order at most $c(A\cup B)$ because of $|B_i|<c(A\cup B)$.  This implies that (III) is satisfied.  
Since $f(B_i)=0$, we must have $f(B'_i) = 0$ or $1$, which means that $|f(B'_i)-\frac{1}{2}|=|f(B_i)-\frac{1}{2}|$ and hence the new partition satisfies (IV).
The new partition satisfies (V) since we have $f(A'_0\cup\dots\cup A'_k)=f(A)$.
However $|T'|=|T|-1$, contradicting minimality of the original tree~$T$ in (VI).

\textbf{Case 2:}  Suppose that $N_r(v_i)\cap B_j\neq \emptyset$ for some $j\neq i$.
Let $C$ be a red component of~$B_j$ which intersects $N_r(v_i)$.
By Claim~\ref{largeB} we have $c(B_j\setminus C)<c(A\cup B)$.

There are two subcases, depending on whether $j\leq l$ holds.

\textbf{Case 2.1:} Suppose that  $j> l$.  By Claim~\ref{onelargecomponentinA} there is a red component $C_A\subseteq A$ of order $c(A\cup B)$.  
 Let $B'_i=B_i\cup C$, $B'_j=(B_j\cup C_A) \setminus C$,  $T'=T$ and $A'_t=A_t\setminus C_A$, $B'_t=B_t$ for all other $t$.  

 The resulting partition trivially satisfies (I).
 Condition (II) follows from the fact that $B_i$ is connected to $C$ by a red edge.
 We have $A'_0\cup\dots\cup A'_k\cup B'_1\cup\dots\cup B'_k=A\cup B$ which implies that the new partition satisfies~(III).  
 Using $|B_i|,|B_j\setminus C|<c(A\cup B)$ we obtain that $f(B'_i)=f(B'_j)\leq 1$ and $f(B'_t)=f(B_t)$ otherwise.  This implies that $\sum_{t=1}^k|f(B'_t)-\frac{1}{2}|=\sum_{t=1}^k|f(B_t)-\frac{1}{2}|$, and so the new partition satisfies (IV).  However, we have $f(A'_0\cup\dots\cup A'_k)= f(A)-1$, contradicting minimality of the original partition in (V).

\textbf{Case 2.2:} Suppose that $j\leq l$.
Since $i\neq j$, this implies that we have $l\geq 2$.

Let $u_i$ be a red neighbour of $v_i$  in $C$.
By  (II), $v_j$ has a red neighbour $u_j$ in $C$.  There must be a red path $P$ between $u_i$ and $u_j$ contained in $C$.

Notice that joining $T$ and $P$ using the edges $u_iv_i$ and $u_jv_j$ produces a graph $T_1$ which has $l-2$ leaves and exactly one cycle (which passes thorough $P$.)
By Claim~\ref{onelargecomponentinA} $A$ contains a red component $C_A$ of order $c(A\cup B)$. By red-connectedness of $K_n$, there must be some edge $xv'_j$ between $x \in T$ and a vertex $v'_j\in C_A$.

We construct a tree $T'$ and sets $A'_t$ and $B'_t$ as follows.
\begin{itemize}
 \item Suppose that $x\neq v_t$ for any $t\in\{1,\dots ,l\}$.  In this case we let $T_2$ be the graph with vertices $T_1 + v'_j$, formed from $T_1$ by adding the edge $xv'_j$.  Notice that $T_2$ has $l-1$ leaves and exactly one cycle.  Therefore, the cycle in $T_2$ must contain a vertex $y$ of degree at least $3$.  Let $v'_i$ be a neighbour of $y$ on the cycle.  We let $T'$ be the tree formed from $T_2$ by removing the edge $yv'_i$.  The leaves of $T'$ are $\{v_1, \dots, v_l\}\setminus\{v_i,v_j\}$, $v'_j$ and possibly $v'_i$ (depending on whether the degree of $v'_i$ in $T_2$ is 2 or not.)
 
 We also let $A'_0=A\cup B_i \cup B_j\setminus P- v'_j$, $B_i=B_j=\emptyset$, and  $A_t=\emptyset$, $B'_t=B_t$, $v'_t=v_t$ for $t\neq i,j$.  

 \item Suppose that $x= v_s$ for some $s\in\{1,\dots ,l\}$ and $f(B_s)=1$.  In this case,  Claim~\ref{largeB} implies that~$B_s$ is connected. Let $v'_s$ be a neighbour of $x$ in $B_s$. 
 Let $T_2$ be the graph with vertices $T_1 + v'_j+ v'_s$, formed from $T_1$ by adding the edges $xv'_j$ and $xv'_s$.
  As before $T_2$ has $l-1$ leaves and exactly one cycle, which contains a vertex $y$ of degree at least $3$.  Let $v'_i$ be a neighbour of $y$ on the cycle.  We let $T'$ be the tree formed from $T_2$ by removing the edge $yv'_i$.  The leaves of $T'$ are $\{v_1, \dots, v_l\}\setminus\{v_i,v_j,v_s\}$, $v'_j$, $v'_s$ and possibly $v'_i$ (depending on whether the degree of $v'_i$ in $T_2$ is 2 or not.)
 
 We also let $A'_0=A\cup B_i \cup B_j\setminus P- v'_j$, $B_i=B_j=\emptyset$, $B'_s=B_s-v'_s$ and  $A_t=\emptyset$, $B'_t=B_t$, $v'_t=v_t$ for $t\neq i,j,s$. 

 \item Suppose that $x= v_s$ for some $s\in\{1,\dots ,l\}$ and $f(B_s)=0$.  
 Let $T_2$ be the graph with vertices $T_1 + v'_j$, formed from $T_1$ by adding the edge $xv'_j$.
  Then $T_2$ has $l-2$ leaves and exactly one cycle, which contains a vertex $y$ of degree at least $3$.  Let $v'_i$ be a neighbour of $y$ on the cycle.  We let $T'$ be the tree formed from $T_2$ by removing the edge $yv'_i$.  The leaves of $T'$ are $\{v_1, \dots, v_l\}\setminus\{v_i,v_j,v_s\}$, $v'_j$ and possibly $v'_i$ (depending on whether the degree of $v'_i$ in $T_2$ is 2 or not.)
 
 We also let $A'_0=A\cup B_i \cup B_j\cup B_s\setminus P- v'_j$, $B_i=B_j=B_s=\emptyset$, and $A_t=\emptyset$, $B'_t=B_t$, $v'_t=v_t$ for  $t\neq i,j,s$.  
\end{itemize}

Clearly the new partition satisfies (I).  
It is easy to see that for all $t$ for which $v'_t$ is defined above, $v'_t$ is connected to all the red components of $B'_t$, so the new partition satisfies (II).

Since $A'_0\cup\dots \cup A'_k\cup B'_1\cup \dots\cup B'_k\subseteq A\cup B$, we must have $c(A'_0\cup\dots \cup A'_k\cup B'_1\cup \dots\cup B'_k)\leq c(A\cup B)$ and hence the new partition satisfies (III).
Since for all $t$, we have $B'_t\subseteq B_t$, the new partition satisfies (IV).
Recall that have $c(B_j\setminus C)<c(A\cup B)$, which combined with the fact that $P$ is nonempty and $|C|\leq c(A\cup B)$ implies that $c(B_j\setminus P)<c(A\cup B)$.  This, combined with the fact that $c(B_i)< c(A\cup B)$ (and, in the third of the above cases, $c(B_s)< c(A\cup B)$) implies that the red components of $A'_1\cup \dots \cup A'_k$ are exactly those of $A$, minus $C_A$.  Therefore we have $f(A'_1\cup \dots \cup A'_k)= f(A)-1$, contradicting minimality of the original partition in (V).
\end{proof}

\begin{claim} \label{BUpperBound}
For every $i$, we have $ |B_i| \leq 2c(A\cup B)$.
\end{claim}

\begin{proof}
Suppose that $B_i> 2c(A\cup B)$.  Combining this with Claim~\ref{onecomponent}, means that there is a red component, $C$ in $B_i$ satisfying $|C|<c(A\cup B)$.
Let $B'_i=B_i\setminus C$,  $A'_0=A_0\cup C$,  and $A'_t=A_t$, $B'_t=B_t$, $T'=T$  otherwise.

The new partition satisfies (I) -- (II) trivially.
It is easy to see that $c(A'_t)=c(A_t)$ and $c(B'_t)=c(B_t)$ for every $t$ which implies that (III) holds for the new partition.
Also we have $f(A'_t)=f(A_t)$ and $f(B'_t)=f(B_t)$ for every $t$ which implies that (IV) -- (V) hold for the new partition.
Since $T'=T$, (VI) holds for the new partition.
Since $|B'_t|=|B_t|$ for $t\neq i$ and $|B'_i|\geq c(A\cup B)$, the new partition satisfies (VII) and (VIII).

However, we have that $|B'_i|<|B_i|$ which contradicts minimality of the original partition in (IX).
\end{proof}

We now prove the theorem.

For each $i=1, \dots, k$ we define a set $N_i\subseteq A\cup B$.
If $i\leq l$, let $N_i=N_r(v_i)$.  If $i>l$, let $N_i=\bigcup_{v\in T} N_r(v)$.

We will show that the graph $K_n\setminus T$, together with the sets $A_0, \dots, A_k$, $B_1, \dots, B_k$, and $N_1, \dots, N_k$ satisfies  conditions (\ref{Con:Partition}) -- (\ref{Con:Last}) of Lemma~\ref{Balancedlemma}.

Condition (\ref{Con:Partition}) follows from the definition of $A_0, \dots, A_k$, and $B_1, \dots, B_k$.  Condition~(\ref{Con:Blue}) follows immediately from (I).
Condition (\ref{Con:N}) follows from (II) whenever $i\leq l$ and from red-connectedness of $K_n$ whenever $i\geq k+1$.
Condition (\ref{Con:A}) follows from the fact that we are assuming (\ref{Aordered}).

Combining Claims~\ref{Abound} and \ref{BLowerBound} implies that  we have $|B_i|+|A_i|\geq c(A\cup B)+ |A_i|\geq |A_0|$ for all $i$. This proves condition (\ref{Con:AB}) of Lemma~\ref{Balancedlemma}.

Combining Claims~\ref{BLowerBound} and ~\ref{BUpperBound} implies that we have $2|B_i|\geq 2c(A\cup B)\geq |B_j|$ for all $i$ and $j$ . This proves condition (\ref{Con:Last}) of Lemma~\ref{Balancedlemma}.

Therefore, the graph  $K_n\setminus T$, together with the sets $A_0, \dots,$ $A_k$, $B_1, \dots,$ $B_k$, and $N_1, \dots,$ $N_k$ satisfies all the conditions of Lemma~\ref{Balancedlemma}.  By Lemma~\ref{Balancedlemma}, $K_n\setminus T$ can be partitioned into paths $P_1, \dots, P_k$ starting in $N_1, \dots, N_k$ and a balanced $(k+1)$-partite graph $H$.  For each $i$, the path $P_i$ can be joined to $T$ to obtain the required partition of $K_n$ into a tree with at most $k$ leaves $T\cup P_1\cup \dots\cup P_k$ and a balanced $(k+1)$-partite graph $H$.
\end{proof}

\section{Ramsey Numbers}\label{SectionRamseyNumbers}
In this section, we use the results of the previous section to determine the the value of the Ramsey number of a path versus certain other graphs.

First we determine  $R(P_n, \kPartite{m}{t})$ whenever $m\equiv 1 \pmod{n-1}$.
\begin{theorem}
If $m\equiv 1 \pmod{n-1}$  then we have
 $$R(P_n, \kPartite{m}{t}) = (t - 1)(n - 1) + t(m - 1) + 1.$$
\end{theorem}
\begin{proof}
For the upper bound, apply Theorem~\ref{PathMultipartite} to the given 2-edge-coloured complete graph on $(t - 1)(n - 1) + t(m - 1) + 1$ vertices.  This gives us $t-1$ red paths and a blue balanced complete $t$-partite graph which, cover all the vertices of $K_{(t - 1)(n - 1) + t(m - 1) + 1}$.  By the Pigeonhole Principle either one of the paths has order at least $n$ or the complete $t$-partite graph has order at least $t(m-1)+1$.  Since the complete $t$-partite graph is balanced, if it has order more than $t(m-1)+1$, then it must have at least $tm$ vertices.  

For the lower bound, consider a colouring of the complete graph on $(t - 1)(n - 1) + t(m - 1)$ vertices consisting of $(t - 1)+ t\ILF{(m - 1)}{(n-1)} $ disjoint red copies of~$K_{n-1}$ and all other edges coloured blue.  The condition $m\equiv 1 \pmod{n-1}$  ensures that we can do this.  Since all the red components of the resulting graph have order at most $n-1$, the graph contains no red $P_n$.   The graph contains no a blue $\kPartite{m}{t}$, since every partition of such a graph would have to intersect at least $(m-1)/(n-1)+1$ of the red copies of $K_{n-1}$ and there are only $(t - 1)(n - 1) + t(m - 1)$ of these.
\end{proof}

In the remainder of this section we will prove Theorem~\ref{PathPowerRamsey}.  First we will use Theorem~\ref{PathMultipartite} and Corollary~\ref{PathMultipartiteConnected} to find upper bounds on $R(P_n, P_m^t)$.
\begin{lemma}\label{WeakRamsey}
The following statements are true.
 \begin{enumerate}[(a)]
  \item $R(P_n, P_m^t)\leq (n-2)t+m$ for all $n, m$ and $t\geq 1$.
  \item Suppose that $t\geq 2$ and $n, m\geq 1$.  Every $2$-edge-coloured complete graph on $(n-1)(t-1)+m$ vertices which is connected in red contains either a red $P_n$ or a blue $P_m^t$.
  \end{enumerate}
\end{lemma}
\begin{proof}
 For part (a), notice that by Theorem~\ref{PathMultipartite}, we can partition a 2-edge-coloured $K_{(n-2)t+m}$ into $t$ red paths $P_1, \dots, P_t$ and a blue $t$th power of a path $P^t$.  Suppose that there are no red paths of order $n$ in  $K_{(n-2)t+m}$.  Suppose that  $i$ of the paths $P_1, \dots, P_t$  are of order $n-1$.    Without loss of generality we may assume that these are the paths $P_1, \dots, P_i$.  We have $|P^t|+(n-~2)(t-i)+(n-1)i\geq|P^t|+|P_1|+\dots+|P_t|=(n-2)t+m$ which implies  $i+|P^t|\geq m$.   For each $j$, let $v_j$ be one of the endpoint of $P_j$.  Notice that since there are no red paths of order $n$ in $K_{(n-2)t+m}$, all the edges in $\{v_1, \dots, v_i, p\}$ are blue for any $p \in P^t$.  This allows us to extend $P^t$ by adding $i$ extra vertices $v_1, \dots, v_i$ to obtain a $t$th power of a path of order $m$.
 
 Part (b) follows immediately from  Corollary~\ref{PathMultipartiteConnected} and the fact that a balanced t-partite graph contains a spanning $(t-1)$st power of a path.
\end{proof}

The following simple lemma allows us to join powers of paths together.
\begin{lemma}\label{Combinepaths}
Let $G$ be a graph.  Suppose that $G$ contains a $(k-i)$th power of a path, $P^{k-i}$, and an $(i-1)$st power of a path, $Q^{i-1}$, such that the following hold.
\begin{enumerate}[\normalfont(i)]
\item All the edges between $P^{k-i}$ and $Q^{i-1}$ are present.
\item $|P^{k-i}|\geq  (k-i+1)\left\lfloor \frac{n}{k+1} \right \rfloor$.
\item $|Q^{i-1}|\geq  i\left\lfloor \frac{n}{k+1} \right \rfloor$.
\item $|P^{k-i}|+|Q^{i-1}|\geq n.$
\end{enumerate}
Then $G$ contains a $k$th power of a path on $n$ vertices.
\end{lemma}
\begin{proof}
Without loss of generality, we may assume that $P^{k-i}$ and $Q^{i-1}$ are the shortest such paths contained in $G$.  We claim that this implies that we have  $|P^{k-i}|+|Q^{i-1}|= n$.  Indeed otherwise (iv) implies that $|P^{k-i}|+|Q^{i-1}|>(k+1)\left\lfloor \frac{n}{k+1} \right \rfloor$, and hence we could remove an endpoint from one of the paths, whilst keeping (ii) and (iii) true.

Let $p_1, \dots, p_{|P^{k-i}|}$ be the vertices of $P^{k-i}$ and $q_1, \dots, q_{|Q^{i-1}|}$ be the vertices of $Q^{i-1}$.  
For convenience set $r_P=|P^{k-i}|\pmod{k-i+1}$ and $r_Q=|q^{i-1}|\pmod{i}$. Together with (ii) and (iii), this ensures that we have  $|P^{k-i}|= r_P + (k-i+1)\left\lfloor\frac{n}{k+1}\right\rfloor$ and $|Q^{i-1}|= r_Q + i\left\lfloor\frac{n}{k+1}\right\rfloor$.  
It is easy to see that the following sequence of vertices is a $k$th power of a path on $n$ vertices.

\begin{align*}
& q_{1},\dots q_{r_Q} \\
& p_1, \dots, p_{k-i+1}, q_{r_Q+1}, \dots, q_{r_Q+i} \\
& p_{k-i+2}, \dots, p_{2(k-i+1)},  q_{r_Q+i+1}, \dots, q_{r_Q+2i} \\
& \ \ \ \ \ \ \ \ \ \ \ \ \ \ \ \vdots \\
& p_{(k-i+1)\left(\left\lfloor\frac{n}{k+1}\right\rfloor-1\right)+1}, \dots, p_{(k-i+1)\left\lfloor\frac{n}{k+1}\right\rfloor},
q_{r_Q+(i-1)\left(\left\lfloor\frac{n}{k+1}\right\rfloor-1\right)+1}, \dots, q_{r_Q+i\left(\left\lfloor\frac{n}{k+1}\right\rfloor-1\right)}\\
& p_{(k-i+1)\left\lfloor\frac{n}{k+1}\right\rfloor+1},\dots,p_{(k-i+1)\left\lfloor\frac{n}{k+1}\right\rfloor+1+r_P}
\end{align*}
\end{proof}

We are now ready to prove Theorem~\ref{PathPowerRamsey}.
\begin{proof}[Proof of Theorem~\ref{PathPowerRamsey}.]
For the lower bound $R(P_n, P_n^k)\geq (n-1)k + \left\lfloor \frac{n}{k+1} \right\rfloor$, consider a colouring of $K_{(n-1)k+\left\lfloor \frac{n}{k+1}  \right\rfloor-1 }$ consisting of $k$ disjoint red copies of $K_{n-1}$ and one disjoint red copy of $K_{\left\lfloor \frac{n}{k+1}  \right\rfloor-1}$.  All edges outside of these are blue.  It is easy to see that when $n\geq k+1$, this colouring contains neither a red path on $n$ vertices nor a blue $P_n^k$.

It remains to prove the upper bound $R(P_n, P_n^k)\leq (n-1)k + \left\lfloor \frac{n}{k+1} \right\rfloor$.
Let $K$ be a  2-edge-coloured complete graph on $(n-1)k + \left\lfloor \frac{n}{k+1} \right\rfloor$ vertices.
Suppose that $K$ does not contain any red paths of order $n$.  We will find a blue copy of $P_n ^k$.
 
 Let $C$ be the largest red component of $K$.
  The following claim will give us three cases to consider.
 \begin{claim} \label{RamseyClaim}
  One of the following always holds.
  \begin{enumerate}[\normalfont(i)]
   \item $|C| \geq  2(n-1) - (k-2) \left\lfloor \frac{n}{k+1} \right\rfloor+1.$
   \item There is a set $B$, such that all the edges between $B$ and $V(K)\setminus B$ are blue and also 
   $$n+ \left\lfloor \frac{n}{k+1} \right\rfloor \leq |B|\leq  2(n-1) - (k-2)\left\lfloor \frac{n}{k+1} \right\rfloor.$$
  \item  The vertices of $K$ can be partitioned into $k$ disjoint sets $B_1, \dots, B_k$ such that for $i \neq j$ all the edges between $B_i$ and $B_j$ are blue and we have 
  \begin{equation*}
  |B_1|\geq |B_2| \geq \dots \geq |B_k| \geq \left\lceil \frac{n}{k+1} \right\rceil.
 \end{equation*}
   \end{enumerate}
 \end{claim}

 \begin{proof}
Suppose that neither (i) nor (ii) hold. 

This implies that all the red components in $K$ have order at most $n+ \left\lfloor \frac{n}{k+1}\right\rfloor-1$.
Let $B$ be a subset of $V(K)$ such that the following hold. 
\begin{enumerate}[(a)]
 \item All the edges between $B$ and $V(K)\setminus B$ are blue.
 \item $|B|\leq n-1+ \left\lfloor \frac{n}{k+1} \right\rfloor$.
 \item $|B|$ is as large as possible.
\end{enumerate}
 Suppose that there is a red component $C'$ in $V(K)\setminus B$ of order at most $\left\lceil\frac{n}{k+1} \right\rceil-1$.  Let $B'=B\cup C'$.  Notice that $n\geq  k\left\lfloor \frac{n}{k+1} \right\rfloor  +  \left\lceil \frac{n}{k+1} \right\rceil$ holds for all integers $n, k \geq 0$.  This implies that we have $|B'|=|B|+ |C'|\leq 2(n-1)-(k-2) \left\lfloor \frac{n}{k+1} \right\rfloor$ thich implies that either $B'$ is a set satisfying (a) and (b) of larger order than $B$, or $B'$ satisfies (ii).
 
  Suppose that all the red components in $V(K)\setminus B$ have order at least $\left\lceil \frac{n}{k+1} \right\rceil$.  Since $n\geq 2$, we have 
  \begin{equation}\label{RamseyClaim1}
  |V(K)\setminus B|\geq (n-1)(k-1)>(k-2)\left(n-1+\left\lfloor \frac{n}{k+1} \right\rfloor\right).
  \end{equation}
  Using the fact that all red components of $K$ have order at most $n-1+\left\lfloor \frac{n}{k+1} \right\rfloor$, (\ref{RamseyClaim1}) implies that $V(K)\setminus B$ must have at least $k-1$ components.  Therefore, the components of $V(K)\setminus B$ can be partitioned into $k-1$ sets $B_2, \dots, B_k$ which, together with $B_1=B$, satisfy~(iii).
  \end{proof}
 
We distinguish three cases, depending on which part of Claim~\ref{RamseyClaim} holds.
 
\textbf{Case 1:} 
 If part (i) of Claim~\ref{RamseyClaim} holds, then there must be some $i \leq k-2$, such that we have 
 \begin{equation} \label{CBounds}
 (k-i)(n-1) - i\left\lfloor \frac{n}{k+1} \right\rfloor+1\leq |C| \leq  (k-i+1)(n-1) - (i-1)\left\lfloor \frac{n}{k+1} \right\rfloor.
 \end{equation}
 
 Combining $ (k-i)(n-1) - i\left\lfloor \frac{n}{k+1} \right\rfloor+1\leq |C|$ with part (b) of Lemma~\ref{WeakRamsey}  shows that $C$ must contain a blue $(k-i)$th power of a path, $P^{k-i}$, on  $n- i \left\lfloor \frac{n}{k+1} \right\rfloor$ vertices.  If $i=0$, then $P^{k-i}$ is a copy of $P_n ^k$, and so the theorem holds.  Therefore, we can assume that $i\geq 1$.
 
  Notice that (\ref{CBounds}) implies that we have $|V\setminus C| \geq (i-1)(n-1) + i \left\lfloor \frac{n}{k+1} \right\rfloor$. 
  Combining this with part (a) of Lemma~\ref{WeakRamsey} shows that $V\setminus C$ must contain a blue $(i-1)$st power of a path, $Q^{i-1}$, on $i\left\lfloor \frac{n}{k+1} \right\rfloor+i-1$ vertices.
  
  Since all the edges between $C$ and $V\setminus C$ are blue we can apply Lemma~\ref{Combinepaths} to $P^{k-i}$ and $Q^{i-1}$ in order to find a blue $k$th power of a path on n vertices in $G$.
  
 \textbf{Case 2:}  Suppose that there is some set $B\subseteq V(K)$ such that all the edges between $B$ and $V(K)\setminus B$ are blue and also 
   $$n+ \left\lfloor \frac{n}{k+1} \right\rfloor \leq |B|\leq  2(n-1) - (k-2)\left\lfloor \frac{n}{k+1} \right\rfloor.$$

 Apply Theorem~\ref{PathRamsey} to $B$ in order to find a path, $P$, of order $2\left\lfloor \frac{n}{k+1} \right\rfloor+2$ in $B$.
 
  Notice that we have $|V(K)\setminus B| \geq  (k-2)(n-1)+(k-1)\left\lfloor \frac{n}{k+1} \right\rfloor$.
  Part (a) of Lemma~\ref{WeakRamsey} shows that $V\setminus B$ must contain a blue $(k-2)$nd power of a path,$Q^{k-2}$ , on $(k-2)\left\lfloor \frac{n}{k+1} \right\rfloor+k-2$ vertices.
  
Since all the edges between $B$ and $V\setminus B$ are blue we can apply Lemma~\ref{Combinepaths}  with $i=k-1$ in order to find a blue $k$th power of a path spanning on $n$ vertices in $G$.
  
 \textbf{Case 3:} Suppose that the vertices of $K$ can be arranged into disjoint sets $B_1, \dots, B_k$ such that for $i \neq j$ all the edges between $B_i$ and $B_j$ are blue and we have 
  \begin{equation*}
  |B_1|\geq |B_2| \geq \dots \geq |B_k| \geq \left\lceil \frac{n}{k+1} \right\rceil.
 \end{equation*}
Let $t$ be the maximum index for which $|B_t|>n-1$.  Notice that $|K|\geq  k(n-1)+\left\lfloor \frac{n}{k+1} \right\rfloor$ implies that we have $|B_1|+\dots+|B_t| -t(n-1)\geq \left\lfloor \frac{n}{k+1} \right\rfloor$.  Therefore, for $i\leq t$, we can choose numbers $x_i$ satisfying $0\leq x_i\leq |B_i|-n+1$ for all $i$ and also $x_1+\dots + x_t= \left\lfloor \frac{n}{k+1} \right\rfloor$.  

For each $i\leq t$ we have $|B_i|=n-1+x_i$, which combined with Theorem~\ref{PathRamsey}, implies that $B_i$ contains a blue path $R_i$ of order $2x_i+1$.  Let $r_{i,0}, r_{i,1}, \dots, r_{i,2x_i}$ be the vertex sequence of~$R_i$.
For each $i\in\{1, \dots, t\}$ and $j\neq i$ choose a set $A_{i,j}$ of vertices in $B_j$  satisfying $|A_{i,j}|=x_i$.   
Note that  for $j>t$, the identity $|B_j|\geq \left\lceil \frac{n}{k+1} \right\rceil$ implies that we have
\begin{equation}\label{eq:morethant}
 |A_{1,j}|+\dots +|A_{t,j}|=\left\lfloor \frac{n}{k+1} \right\rfloor\leq |B_j| .
 \end{equation}
For $j\leq t$, the identities $|B_j|\geq n$ and $x_j\leq  \left\lfloor \frac{n}{k+1} \right\rfloor$ imply that we have 
 \begin{equation}\label{eq:lessthant}
  |A_{1,j}|+\dots +|A_{j-1,j}|+|R_j|+|A_{j+1,j}|+\dots+|A_{t,j}|= \left\lfloor \frac{n}{k+1} \right\rfloor+x_j+1\leq |B_j|.
\end{equation}
Now, (\ref{eq:morethant}) and (\ref{eq:lessthant}) imply that we can choose the sets $A_{i,j}$, such that $A_{i,j}$ and $A_{i',j}$ are disjoint for $i\neq i'$.  In addition, for every $j\leq t$, (\ref{eq:lessthant}) implies that we can choose the sets $A_{i,j}$ to be disjoint from $R_j$.  Let $a_{i,j,1}, \dots, a_{i,j,x_i}$ be the vertices of $A_{i,j}$.  If $n\not\equiv 0 \pmod{k+1}$, then the inequalities in both (\ref{eq:morethant}) and (\ref{eq:lessthant}) must be strict, and so there must be at least one vertex contained in $B_i$ outside of $R_i \cup A_{i,1} \cup \dots \cup A_{i,t}$.  Let $b_i$ be this vertex.

For $i=1, \dots, t$ and $j = 1, \dots, x_i$, we will define blue paths $P_{i,j}$ of order $k+1$ as follows.
If $i=1$ and $j\in\{1, \dots, x_1-1\}$, then $P_{i,j}$ has the following vertex sequence.
$$P_{1,j}= r_{1,2j-1}, r_{1,2j}, a_{1,2,j}, a_{1,3,j}, \dots, a_{1,k,j}.$$

If $i=1$ and $j= x_1$, then $P_{i,j}$ has the following vertex sequence.
$$P_{1,x_1}= r_{1,2x_1-1}, r_{1,2x_1}, r_{2,0},  a_{1,3,x_1}, \dots, a_{1,k,x_1}.$$

If $i\in\{2, \dots, t-1\}$ and $j\in\{1, \dots, x_i-1\}$, then $P_{i,j}$ has the following vertex sequence.
$$P_{i,j}= r_{i,2j-1}, a_{i,1,j}, a_{i,2,j}, \dots, a_{i,i-1,j}, r_{i,2j}, a_{i,i+1,j}, a_{i,i+2,j},\dots, a_{i,k,j}.$$

If $i\in\{2, \dots, t-1\}$ and $j= x_i$, then $P_{i,j}$ has the following vertex sequence.
$$P_{i,x_i}= r_{i,2x_i-1}, a_{i,1,x_i}, a_{i,2,x_i}, \dots, a_{i,i-1,x_i}, r_{i,2x_i}, r_{i+1,0} , a_{i,i+2,x_i},\dots, a_{i,k,x_i}.$$

If $i=t$ and $j\in \{1, \dots, x_t\}$, then $P_{i,j}$ has the following vertex sequence.
$$P_{t,j}= r_{t,2j-1},a_{t,1,j} ,  a_{t,2,j}, \dots, a_{t,t-1,j}, r_{t,2j}.$$

If $n\not\equiv 0 \pmod{k+1}$, we also define a path $P_0$ of order $k$ with vertex sequence
$$P_0=r_{1,0}, b_2, b_3, \dots, b_k.$$
If $n\equiv 0 \pmod{k+1}$, let $P_0=\emptyset$.

Notice that the paths $P_{i,j}$ and $P_{i',j'}$ are disjoint for $(i,j)\neq (i',j')$.  Similarly $P_0$ is disjoint from all the paths $P_{i,j}$.  
We have the following
\begin{equation}\label{PathsOrder}
|P_0|+ \sum_{i=1}^k \sum_{j=1}^{x_i} |P_{i,j}|= |P_0| + (k+1)(x_1+ \dots + x_k)= |P_0|+(k+1)\left\lfloor \frac{n}{k+1} \right\rfloor\geq n.
\end{equation}

We claim that the following path is in fact a blue $k$th power of a path.
$$
P=\begin{cases} 
\begin{array}{c}
P_0+ \\
P_{1,1}+ P_{1,2}+ \dots + P_{1,x_t}+\\
P_{2,1}+ P_{2,2}+ \dots + P_{2,x_{2}}+\\
\vdots\\
P_{t,1}+ P_{t,2}+ \dots + P_{t,x_t}. 
\end{array}
\end{cases}
$$
To see that $P$ is a $k$th power of a path one needs to check that any pair of vertices $a,b$ at distance at most $k$ along $P$ are connected by a blue edge.
It is easy to check that for any such $a$ and $b$, either  $a \in B_i$ and $b\in B_j$ for some $i\neq j$ or $a$ and $b$ are consecutive vertices along $P_0$ or $P_{i,j}$ for some $i,j$.  In either case $ab$ is blue implying that $P$ is a blue $k$th power of a path.

The identity (\ref{PathsOrder}) shows that $|P|\geq n$, completing the proof.
\end{proof}

\section{Remarks}\label{SectionRamseyRemarks}
In this section we dicuss some further directions one might take with the results presented in this paper.  

\begin{itemize}
 \item 
It would be interesting to see if there are any other Ramsey numbers which can be determined using the techniques we used in this paper.  

If $G$ is a graph of (vertex)-chromatic number $\chi(G)$, then $\sigma(G)$ is defined to be the smallest possible order of a colour class in a proper $\chi(G)$-vertex colouring of~$G$.  Generalising a construction of Chvatal and Harary, Burr \cite{Burr} showed that if $H$ is a graph and  $G$ is a connected graph and satisfying $|G|\geq \sigma(H)$, then we have
\begin{equation}\label{eqBurr}
R(G,H)\geq (\chi(H)-1)(|G|-1) +\sigma(H)
\end{equation}
This identity comes from considering a colouring consisting of $\chi(H)-1$ red copies of $K_{|G|-1}$ and one red copy of $K_{\sigma(H)-1}$.  Notice that for a $k$th power of a path, we have $\chi(P_n^k)=k+1$ and $\sigma(P_n^k)=\left\lfloor\frac{n}{k+1}\right\rfloor$.  Therefore, Theorem~\ref{PathPowerRamsey} shows that (\ref{eqBurr}) is best possible when $G=P_n$ and $H=P^k_n$.  

It is an interesting question to find other pairs of graphs for which equality holds in~(\ref{eqBurr}) (see \cite{Skokan, Rousseau}).  Allen, Brightwell, and Skokan   conjectured that when $G$ is a path, then equality holds in~(\ref{eqBurr}) for any graph $H$ satisfying $|G|\geq \chi(H)|H|$.
\begin{conjecture}[Allen, Brightwell, and Skokan] \label{SkokanConjecture} 
For every graph $H$, $R(P_n,H)=(\chi(H)-1)(n-1) +\sigma(H)$ whenever $n \geq \chi(H)|H|$.
\end{conjecture}
It is easy to see that in order to prove Conjecture~\ref{SkokanConjecture}, it is sufficient to prove it only in the case when $H$ is a (not necessarily balanced) complete multipartite graph.

The techniques used in this paper look like they may be useful in approaching Conjecture~\ref{SkokanConjecture}.  One reason for this is that several parts of the proof of Theorem~\ref{PathPowerRamsey} would have worked if we were looking for the Ramsey number of a path versus a balanced complete multipartite graph insead of a power of a path.

\item
Recall that Lemma~\ref{PathBipartite} only implies part of H\"aggkvist result (Theorem~\ref{Haggkvist}).  However, it is easy to prove an ``unbalanced'' version of Lemma~\ref{PathBipartite} which implies  Theorem~\ref{Haggkvist}.
\begin{lemma}\label{PathBipartiteUnbalanced}
 Suppose that the edges of $K_n$ are coloured with 2 colours and we have an integer $t$ satisfying $0\leq t\leq n$.  Then there is a partition of $K_n$ into a red path and a blue copy of $K_{m, m+t}$ for some integer $m$.
\end{lemma}
The proof of this lemma is nearly identical to the one we gave of Lemma~\ref{PathBipartite} in the Section~\ref{SectionRamseyNotation}.  Indeed, the only modification that needs to be made is that we need to add the condition ``$\big||X|-|Y|\big|\geq t$'' on the sets $X$ and $Y$ in the proof of Lemma~\ref{PathBipartite}.

\item
It would be interesting to see whether Theorems~\ref{PathMultipartite} and~\ref{TreeMultipartite} have any applications in the area of partitioning coloured complete graphs.  In particular, given that Lemma~\ref{PathBipartite} played an important role in the proof of the $r=3$ case of Conjecture~\ref{Gyarfas} in \cite{PokrovskiyCycles}, it is possible that Theorems~\ref{PathMultipartite} and~\ref{TreeMultipartite} may help with that conjecture.

Classically, results about partitioning coloured graphs would partition a graph into monochromatic subgraphs which all have the same structure.  For example Theorems~\ref{GerencserGyarfas} and~\ref{GyarfasLehel} partition graphs into monochromatic paths.  Lemma~\ref{PathBipartite} and Theorem~\ref{TreeMultipartite} stand out from these since they partition a 2-edge-coloured complete graph into two monochromatic subgraphs which have very different structure.  It would be interesting to find other natural results along the same lines.  Some results about partitioning a 2-edge-coloured complete graph into a monochromatic cycle and a monochromatic graph with high minimum degree will appear in~\cite{PokrovskiyCycleSparseGraph}.
\end{itemize}

\bigskip\noindent
\textbf{Acknowledgment}

\smallskip\noindent
The author would like to thank his supervisors Jan van den Heuvel and Jozef Skokan for their advice and discussions.

\bibliography{pathpartition}
\bibliographystyle{abbrv}
\end{document}